\documentclass[reqno]{amsart}
\usepackage{amssymb}
\usepackage[colorlinks=true]{hyperref}
\hypersetup{urlcolor=blue, citecolor=red}
\usepackage{color}

\newtheorem{Theorem}{Theorem}[section]

\newtheorem{Lemma}[Theorem]{Lemma}
\newtheorem{Proposition}[Theorem]{Proposition}
\newtheorem{Definition}[Theorem]{Definition}
\newtheorem{Remark}[Theorem]{Remark}

\def\C {\mathbb C}

\def\H {\mathbb H}

\def\R {\mathbb R}

\newcommand{\<}{\langle}
\renewcommand{\>}{\rangle}
\renewcommand{\(}{\left(}
\renewcommand{\)}{\right)}
\renewcommand{\Re}{\operatorname{Re}}
\renewcommand{\Im}{\operatorname{Im}}

\newcommand{\lesim}{\lesssim}
\renewcommand{\div}{\operatorname{div}}

\newcommand{\de}[2]{\frac{\partial #1}{\partial #2}}
\newcommand{\p}{\partial}
\newcommand{\Vol}{\operatorname{Vol}}

\begin{document}
\title[Polyharmonic operator on admissible manifolds]{Determining the first order perturbation of a polyharmonic operator on admissible manifolds}

\author[Yernat M. Assylbekov]{Yernat M. Assylbekov}
\address{Department of Mathematics, University of Washington, Seattle, WA 98195-4350, USA}
\email{y\_assylbekov@yahoo.com}

\author[Yang Yang]{Yang Yang}
\address{Department of Mathematics, Purdue University, West Lafayette, IN 47907, USA}
\email{yang926@purdue.edu}

\begin{abstract}
We consider the inverse boundary value problem for the first order perturbation of the polyharmonic operator $\mathcal L_{g,X,q}$, with $X$ being a $W^{1,\infty}$ vector field and $q$ being an $L^\infty$ function on compact Riemannian manifolds with boundary which are conformally embedded in a product of the Euclidean line and a simple manifold. We show that the knowledge of the Dirichlet-to-Neumann determines $X$ and $q$ uniquely. The method is based on the construction of complex geometrical optics solutions using the Carleman estimate for the Laplace-Beltrami operator due to Dos Santos Ferreira, Kenig, Salo and Uhlmann. Notice that the corresponding uniqueness result does not hold for the first order perturbation of the Laplace-Beltrami operator.
\end{abstract}

\maketitle

\section{Introduction}
Let $(M,g)$ be a compact oriented Riemannian smooth manifold with boundary. Throughout this paper, the word ``smooth" will be used as the synonym of ``$C^\infty$". Let $\Delta_g$ be the Laplace-Beltrami operator associated to the metric $g$ which is given in local coordinates by
$$
\Delta_g u=|g|^{-1/2}\de{}{x^j}\(|g|^{1/2} g^{jk}\de{u}{x^k}\),
$$
where as usual $(g^{jk})$ is the matrix inverse of $(g_{jk})$, and $|g|=\det(g_{jk})$. If $F$ denotes a function or distribution space ($C^k$, $L^p$ , $H^k$, $\mathcal D'$, etc.), then we will denote by $F(M,TM)$ the corresponding space of vector fields on $M$.

Let $X\in W^{1,\infty}(M,TM)$ and $q\in L^\infty(M)$. Consider the polyharmonic operator $(-\Delta_g)^m$, $m\ge 1$, with the first order perturbation induced by $X$ and $q$
$$
\mathcal{L}_{g,X,q}=(-\Delta_g)^m+X+q
$$
The operator $\mathcal L_{g,X,q}$ equipped with the domain
$$
\mathcal{D}(\mathcal L_{g,X,q})= \{u\in H^{2m}(M): \gamma u=0\}=H^{2m}(M)\cap H^m_0(M)
$$
is an unbounded closed operator on $L^2(M)$ with purely discrete spectrum; see \cite{Grubb}. Here and in what follows,
$$
\gamma u:=(u|_{\p M},\Delta_g u|_{\p M},\dots,\Delta_g^{m-1} u|_{\p M})
$$
is the Dirichlet trace of $u$, and $H^s(M)$ is the standard Sobolev space on $M$, $s\in\R$.
\\

We make the assumption that $0$ is not a Dirichlet eigenvalue of $\mathcal L_{g,X,q}$ in $M$. Under this assumption, for any $f=(f_0,\dots,f_{m-1})\in \mathcal H_m(\p M):=\prod_{j=0}^{m-1}H^{2m-2j-1/2}(\p M)$, the Dirichlet problem
\begin{equation}\label{Dirichlet problem}
\begin{aligned}
\mathcal L_{g,X,q}u&=0\quad\text{in}\quad M,\\
\gamma u&=f\quad\text{in}\quad\p M,
\end{aligned}
\end{equation}
has a unique solution $u\in H^{2m}(M)$. Let $\nu$ be an outer unit normal to $\p M$. Introducing the Neumann trace operator $\widetilde{\gamma}$ by
\begin{align*}
&\widetilde{\gamma}: H^{2m}(M)\to \prod_{j=0}^{m-1}H^{2m-2j-3/2}(\p M),\\
&\widetilde{\gamma}u=(\p_\nu u|_{\p M},\p_\nu \Delta_g u|_{\p M},\dots,\p_\nu \Delta_g^{m-1} u|_{\p M}),
\end{align*}
we define the Dirichlet-to-Neumann map $N_{g,X,q}$ by
$$
N_{g,X,q}:\mathcal H_m(\p M)\to \prod_{j=0}^{m-1}H^{2m-2j-3/2}(\p M),\quad
N_{g,X,q}(f)=\widetilde{\gamma}u,
$$
where $u\in H^{2m}(M)$ is the unique solution to the boundary value problem \eqref{Dirichlet problem}. Let us also introduce the set of the Cauchy data for the operator $\mathcal L_{g,X,q}$
$$
C_{g,X,q}=\{(\gamma u,\widetilde{\gamma}u):u\in H^{2m}(M),\quad \mathcal L_{g,X,q}u=0\}.
$$
When $0$ is not a Dirichlet eigenvalue of $\mathcal L_{g,X,q}$ in $M$, the set $C_{g,X,q}$ is the graph of the Dirichlet-to-Neumann map $N_{g,X,q}$.\\

The inverse problem we are concerned in this paper is to recover the vector field $X$ and the function $q$ from the knowledge of the Dirichlet-to-Neumann map $N_{g,X,q}$ on the boundary $\p M$.

When $m=1$, the Dirichlet-to-Neumann map $N_{g,X,q}$ is invariant under gauge transformations in the following sense. Let $\psi$ be a $C^2(M)$ such that $\psi|_{\p M}=0$ and $\p_\nu \psi|_{\p M}=0$. Then
$$
e^{-i\psi}\mathcal L_{g,X,q}e^{i\psi}=\mathcal L_{g,\widetilde X,\widetilde q},\quad N_{g,\widetilde X,\widetilde q}=N_{g,X,q},
$$
where
$$
\widetilde X=X+2\nabla \psi,\quad \widetilde q=q+\<X,\nabla \psi\>_g+|\nabla \psi|_g^2-i\Delta_g \psi.
$$
Therefore, we may hope to recover $X$ and $q$ from boundary measurements only modulo the above gauge transformations.

In the Euclidean setting, this inverse boundary value problem has been extensively studied, usually in the context of magnetic Schr\"odinger operators \cite{KU,NSU,S,Sun,Tol}. In the case of Riemannian manifolds, this was proved in \cite{DKSU} for the special class of so-called admissible manifolds.\\

Let us now introduce admissible manifolds. For this we need the notion of simple manifolds \cite{Shar}. The notion of simplicity arises naturally in the context of the boundary rigidity problem \cite{Michel}.

\begin{Definition}{\rm
A compact Riemannian manifold $(M,g)$ with boundary is said to be \emph{simple} if the boundary $\p M$ is strictly convex, and for any point $x\in M$ the exponential map $\exp_x$ is a diffeomorphism from its maximal domain in $T_x M$ onto $M$.
}\end{Definition}

\begin{Definition}{\rm
A compact Riemannian manifold $(M,g)$ with boundary of dimension $n\ge 3$, is said to be \emph{admissible} if it is conformal to a submanifold with boundary of $\R\times (M_0,g_0)$ where $(M_0,g_0)$ is a simple $(n-1)$-dimensional manifold.
}\end{Definition}

Examples of admissible manifolds include the following:

1. Bounded domains in Euclidean space, in the sphere minus a point, or in hyperbolic space. In the last two cases, the manifold is conformal to a domain in Euclidean space via stereographic
projection.

2. More generally, any domain in a locally conformally flat manifold is admissible, provided that the domain is appropriately small. Such manifolds include locally symmetric $3$-dimensional
spaces, which have parallel curvature tensor so their Cotton tensor vanishes (see the \cite[Appendix~B]{DKSU}).

3. Any bounded domain $M$ in $\R^n$, endowed with a metric which in some coordinates has the form
$$
g(x_1,x')=c(x)\(\begin{matrix}
1&0\\
0&g_0(x')
\end{matrix}\),
$$
with $c>0$ and $g_0$ simple, is admissible.

4. The class of admissible metrics is stable under $C^2$-small perturbations of $g_0$.
\medskip

It was shown in \cite{KLU} that, in the Euclidean case, the obstruction to uniqueness coming from the gauge equivalence when $m=1$ can be eliminated by considering operators of higher order. The purpose of this paper is to extend this result for the case of admissible manifolds.
\medskip

Our main result is as follows.
\begin{Theorem}\label{main th}
Let $(M,g)$ be admissible, and let $m\ge 2$ be an integer. 
Suppose that $X_1,X_2\in W^{1,\infty}(\R\times M_0,T(\R\times M_0))\cap \mathcal E'(M,TM)$ 
and $q_1,q_2\in L^\infty(M)$ are such that $0$ is not a Dirichlet eigenvalue of $\mathcal L_{g,X_1,q_1}$ and $\mathcal L_{g,X_2,q_2}$ in $M$. If $N_{g,X_1,q_1}=N_{g,X_2,q_2}$, then $X_1=X_2$ and $q_1=q_2$.
\end{Theorem}

The key ingredient in the proof of Theorem \ref{main th} is the construction of complex geometric optics solutions for the operator $\mathcal L_{g,X,q}$ with $X$ being a $W^{1,\infty}$ vector field and $q$ an $L^\infty(M)$ function. For this, we use the method of Carleman estimates which is based on the corresponding Carleman estimate for the Laplacian due to Dos Santos Ferreira, Kenig, Salo and Uhlmann \cite{DKSU}.\\

In Theorem \ref{main th}, the condition that $X_1=X_2=0$ on $\partial M$ is needed to extend the vector fields $X_1$ and $X_2$ to a slightly larger simple manifold than $M$ while preserving the $W^{1,\infty}$ regularities. When more regularities on $X_j$ and $q_j$ ($j=1,2$) are available, we can show a boundary determination result for the vector fields and thus drop such an assumption. This is the following theorem.
\begin{Theorem}\label{main th2}
Let $(M,g)$ be admissible, and let $m\ge 2$ be an integer. 
Suppose that $X_1,X_2\in C^{\infty}(M,TM)$ 
and $q_1,q_2\in C^\infty(M)$ are such that $0$ is not a Dirichlet eigenvalue of $\mathcal L_{g,X_1,q_1}$ and $\mathcal L_{g,X_2,q_2}$ in $M$. If $N_{g,X_1,q_1}=N_{g,X_2,q_2}$, then $X_1=X_2$ and $q_1=q_2$.
\end{Theorem}

Let $\pi:\mathbb{R}\times M_0 \rightarrow M_0$ be the canonical projection $\pi(x_1,x')=x'$. It is interesting to notice that the boundary determination becomes unnecessary if $(\pi(M),g_0)$ is a simple $(n-1)$-dimensional manifold and $\partial M$ is connected.
\begin{Theorem}\label{main th3}
Let $(M,g)$ be admissible, and let $m\ge 2$ be an integer. 
Suppose that $X_1,X_2\in W^{1,\infty}(M,TM)$ 
and $q_1,q_2\in L^\infty(M)$ are such that $0$ is not a Dirichlet eigenvalue of $\mathcal L_{g,X_1,q_1}$ and $\mathcal L_{g,X_2,q_2}$ in $M$. Suppose further that $(\pi(M),g_0)$ is a simple $(n-1)$-dimensional manifold and $\partial M$ is connected. If $N_{g,X_1,q_1}=N_{g,X_2,q_2}$, then $X_1=X_2$ and $q_1=q_2$.
\end{Theorem}

In the case of Euclidean space, the recovery of a zeroth order perturbation of the biharmonic operator, that is when $m=2$, has been studied by Isakov \cite{Isakov}, where a uniqueness result was obtained, similarly to the case of the Schr\"odinger operator. The recovery of a first order perturbation of the biharmonic operator from partial data was studied in \cite{KLU2} in a bounded domain, and in \cite{Y} in an infinite slab. Higher order operators occur in the areas of physics and geometry such as the study of the Kirchhoff plate equation in the theory of elasticity, and the study of the Paneitz-Branson operator in conformal geometry; for more details see \cite{GGS}.

Finally, we would like to remark that the problem considered in this paper can be viewed as generalization of the Calder\'on's inverse conductivity problem \cite{Cal}, known also as electrical impedance tomography. In the fundamental paper by Sylvester and Uhlmann \cite{SyU} it was shown that $C^2$ conductivities can be uniquely determined from boundary measurements. A corresponding result was proved by Dos~Santos~Ferreira, Kenig, Salo and Uhlmann \cite{DKSU} in the setting of admissible geometries.

The structure of the paper is as follows. In Section 2 a Carleman estimate is derived for polyharmonic operators based on a similar estimate for the Laplace-Beltrami operator. Section~\ref{section on CGOs} is devoted to the construction of  complex geometric optics solutions for the perturbed polyharmonic operator $\mathcal L_{g,X,q}$ with $X$ being a $W^{1,\infty}$ vector field and $q\in L^\infty(M)$. Then the proof of Theorem~\ref{main th} is given in Section~\ref{Proof}. Attenuated ray transform is the subject of Section~\ref{Attenuated ray transform}. In Section~\ref{section on boundary determination}, we show that the Dirichlet-to-Neumann map determines $X$ on the boundary, this leads to the proof of Theorem \ref{main th2}. Finally, the proof of Theorem \ref{main th3} is given in Section 7.

\section{Carleman estimates for polyharmonic operators}
Let $(M,g)$ be a Riemannian manifold with boundary. In this section, following \cite{DKSU,KSU}, we shall use the method of Carleman estimates to construct complex geometric optics solutions for the equation $\mathcal{L}_{g,X,q}u=0$ in $M$, with $X$ being a $W^{1,\infty}$ vector field on $M$ and $q\in L^\infty(M)$.

We start by recalling the definition of the Carleman weight for the semiclassical Laplace-Beltrami operator $-h^2\Delta_g$. Let $U$ be an open manifold without boundart such that $M\subset\subset  U$ and let $\varphi\in C^\infty(U,\R)$. Consider the conjugated operator
$$
P_\varphi=e^{\varphi/h}(-h^2\Delta_g)e^{-\varphi/h}.
$$
Following \cite{DKSU,KSU}, we say that $\varphi$ is a limiting Carleman weight for $-h^2\Delta_g$ in $U$, if it has non-vanishing differential, and if it satisfies the Poisson bracket condition
$$
\{\overline{p_\varphi},p_\varphi\}(x,\xi)=0\quad\text{when}\quad p_\varphi(x,\xi)=0,\quad(x,\xi)\in T^*M,
$$
where $p_\varphi$ is the semiclassical principal symbol of $P_\varphi$.

First we shall derive a Carleman estimate for the semiclassical polyharmonic operator $(-h^2\Delta_g)^m$, where $h>0$ is a small parameter, by iterating the corresponding Carleman estimate for the semiclassical Laplace-Beltrami operator $-h^2\Delta_g$, which we now proceed
to recall the following \cite{DKSU,KSU}.

We use the notation $d\Vol_g$ for the volume form of $(M,g)$. For any two functions $u,v$ on $M$, define an inner product
$$
(u|v):=\int_M u(x)\overline{v(x)}\,d\Vol_g(x),
$$
and the corresponding norm will be denoted by $\|\cdot\|_{L^2(M)}$. We also write for short
$$
\|\nabla u\|_{L^2(M)}=\||\nabla u|\|_{L^2(M)}=\(\int_M |\nabla u(x)|_g^2\,d\Vol_g(x)\)^{1/2}.
$$

We assume that $(M,g)$ is embedded in a compact manifold $(N,g)$ without boundary, and $\varphi$ is a limiting Carleman weight on $(U,g)$, where $U$ is an open submanifold of $N$ such that $M\subset \subset U$. By semiclassical spectral theorem one can define for $s\in\R$ the semiclassical Bessel potentials $J^s=(1-h^2\Delta_g)^{s/2}$. One has $J^s J^t=J^{s+t}$, and Bessel potentials commute with any function of $-\Delta_g$. Define for $s\in\R$ the semiclassical Sobolev space associated to the norm
$$
\|u\|_{H^s_{\rm scl}(N)}=\|J^s u\|_{L^2(N)}.
$$

Our starting point is the following Carleman estimate for the semiclassical Laplace-Beltrami
operator $-h^2\Delta_g$ which is due to Dos Santos Ferreira, Kenig, Salo and Uhlmann \cite[Lemma~4.3]{DKSU}. In what follows, $A\lesim B$ means that $A\le CB$ where $C > 0$ is a constant independent of $h$ and $A, B$.
\begin{Proposition}\label{Carleman estimate for LB}
Let $(U,g)$ be an open Riemannian manifold and $(M,g)$ be a smooth compact Riemannian submanifold with boundary such that $M\subset\subset U$. Let $\varphi$ be a limiting Carleman weight on $(U,g)$. Then for all $h>0$ small enough and $s\in \R$, we have
$$
h\|u\|_{H^{s+1}_{\rm scl}(N)}\lesim \|e^{\varphi/h}(-h^2\Delta_g)e^{-\varphi/h}u\|_{H^s_{\rm scl}(N)}
$$
for all $u\in C^\infty_0(M)$.
\end{Proposition}

Next we shall derive a Carleman estimate for the operator $\mathcal L_{g,X,q}$ with $X$ being a $W^{1,\infty}$ vector field on $M$ and $q\in L^\infty(M)$. To that end we shall use Proposition~\ref{Carleman estimate for LB} with $s=-1$. We have the following result.

\begin{Proposition}\label{Carleman est L_{g,X,q} prop}
Let $(U,g)$ be an open Riemannian manifold and $(M,g)$ be a smooth compact Riemannian submanifold with boundary such that $M\subset\subset U$. Let $\varphi$ be a limiting Carleman weight on $(U,g)$. Suppose that $X$ is a $W^{1,\infty}$ vector field on $M$ and $q\in L^\infty(M)$. Then for all $h>0$ small enough, we have
\begin{equation}\label{Carleman est L_{g,X,q}}
\|u\|_{L^2(N)}\lesim \frac{1}{h^m}\|e^{\varphi/h}(h^{2m}\mathcal{L}_{g,X,q})e^{-\varphi/h}u\|_{H^{-1}_{\rm scl}(N)},
\end{equation}
for all $u\in C^\infty_0(M)$.
\end{Proposition}
\begin{proof}
Iterating the Carleman estimate in Proposition~\ref{Carleman est L_{g,X,q} prop} $m$ times, $m\ge2$, we get the following Carleman estimate for the polyharmonic operator,
$$
h^m\|u\|_{H^{s+m}_{\rm{scl}}(N)}\lesim\|e^{\varphi/h}(-h^2\Delta_g)^m e^{-\varphi/h}u\|_{H^{s}_{\rm{scl}}(N)},
$$
for all $u\in C^\infty_0(\Omega)$, $s\in\R$ and $h>0$ small enough. We shall use this estimate with $s=-1$:
\begin{equation}\label{Carleman for (-h^2Delta)^m 1}
h^m\|u\|_{H^{m-1}_{\rm{scl}}(N)}\lesim \|e^{\varphi/h}(-h^2\Delta_g)^m e^{-\varphi/h}u\|_{H^{-1}_{\rm scl}(N)},
\end{equation}
for all $u\in C^\infty_0(N)$ and $h>0$ small enough. Since we are dealing with first order
perturbations of the polyharmonic operator and $m\ge 2$, the following weakened version of \eqref{Carleman for (-h^2Delta)^m 1} will be sufficient for our purposes
\begin{equation}\label{Carleman for (-h^2Delta)^m}
h^m\|u\|_{L^2(N)}\lesim \|e^{\varphi/h}(-h^2\Delta_g)^m e^{-\varphi/h}u\|_{H^{-1}_{\rm scl}(N)},
\end{equation}
for all $u\in C^\infty_0(M)$ and $h>0$ small enough.

It is easy to see that
\begin{equation}\label{q part}
\|e^{\varphi/h}h^{2m}qe^{-\varphi/h}u\|_{L^2(N)}\lesim h^{2m}\|q\|_{L^\infty(M)}\|u\|_{H^{1}_{\rm{scl}}(N)}.
\end{equation}

Note that $e^{\varphi/h}h^{2m}X(e^{-\varphi/h}u)=-h^{2m-1}\<X,\nabla\varphi\>_gu+h^{2m}\<X,\nabla u\>_g$. Therefore, since $m\ge 2$
\begin{align*}
\|h^{2m-1}\<X,\nabla\varphi\>_gu\|_{L^2(N)}&\le h^{2m-1}\|\<X,\nabla\varphi\>_g\|_{L^\infty(M)}\|u\|_{L^2(N)}\\
&\le h^{m}\|\<X,\nabla\varphi\>_g\|_{L^\infty(M)}\|u\|_{H^1_{\rm scl}(N)}
\end{align*}
and
$$
\|h^{2m}\<X,h\nabla u\>_g\|_{L^2(N)}\le h^{2m-1}\|X\|_{L^\infty(M)}\|u\|_{H^1_{\rm scl}(N)}\le h^{m}\|X\|_{L^\infty(M)}\|u\|_{H^1_{\rm scl}(N)}.
$$
imply
$$
\|e^{\varphi/h}h^{2m}X(e^{-\varphi/h}u)\|_{L^2(N)}\lesim h^m \|u\|_{H^1_{\rm scl}(N)}.
$$
Combining this together with estimates \eqref{Carleman for (-h^2Delta)^m} and \eqref{q part}, we get the result.
\end{proof}

Set
$$
\mathcal L_{\varphi}:=e^{\varphi/h}(h^{2m}\mathcal L_{g,X,q})e^{-\varphi/h}.
$$
Then we have
$$
\<\mathcal L_{\varphi} u,\overline{v}\>_{\Omega}= \<u,\overline{\mathcal L_{\varphi}^* v}\>_{\Omega},\quad u,v\in C^\infty_0(\Omega),
$$
where $\mathcal L_{\varphi}^*=e^{-\varphi/h}(h^2\mathcal L_{g,-X,-\div_g X+{q}})e^{\varphi/h}$ is the formal adjoint of $\mathcal L_{\varphi}$, and $\langle \cdot,\cdot\rangle_M$ is the distribution duality on $M$. The estimate in Proposition~\ref{Carleman est L_{g,X,q} prop} holds for $\mathcal L_{\varphi}^*$, since $-\varphi$ is a limiting Carleman weight as well.

To construct the complex geometric optics solutions for the operator $\mathcal L_{g,X,q}$, we need to convert the Carleman estimate \eqref{Carleman est L_{g,X,q}} for $\mathcal L_{\varphi}^*$ into the following solvability result. The proof is essentially well-known, and we include it here for the convenience of the reader. We shall use the following notation for the semiclassical Sobolev norm on $M$
$$
\|u\|_{H^{1}_{\rm scl}(M)}^2=\|u\|_{L^2(M)}^2+\|h\nabla u\|_{L^2(M)}^2.
$$

\begin{Proposition}\label{Solvability result}
Let $X$ be a $W^{1,\infty}$ vector field on $M$ and $q\in L^\infty(M)$ and assume that $m\ge 2$. If $h>0$ is small enough, then for any $v\in L^2(M)$ there is a solution $u\in H^1(M)$ of the equation
$$
e^{\varphi/h}h^{2m}\mathcal{L}_{g,X,q}e^{-\varphi/h}u=v
$$
satisfying
$$
\|u\|_{H^1_{\rm{scl}}(M)}\le \frac{C}{h^m}\|v\|_{L^2(M)}.
$$
\end{Proposition}
\begin{proof}
Let $v\in H^{-1}(M)$ and let us consider the following complex linear functional,
$$
L: \mathcal L_{\varphi}^* C_0^\infty(M)\to \C, \quad \mathcal L_{\varphi}^* w \mapsto \langle w, \overline{v}\rangle_M.
$$
By the Carleman estimate \eqref{Carleman est L_{g,X,q}} for $\mathcal L_{\varphi}^*$, the map $L$ is well-defined.  
Let $w\in C_0^\infty(M)$. Then we have
\begin{align*}
|L(\mathcal L_{\varphi}^* w)|=|\langle w, \overline{v}\rangle_M|&\le \|w\|_{L^2(N)}\|v\|_{L^2(M)}\\
&\lesim \frac{1}{h^m}\|v\|_{L^2(M)}\|\mathcal L_{\varphi}^* w\|_{H^{-1}_{\rm scl}(N)}.
\end{align*}
By the Hahn-Banach theorem, we may extend $L$ to a linear continuous functional $\tilde L$ on $L^2(N)$, without increasing its norm. 
By the Riesz representation theorem, there exists $u\in H^1(\R^n)$ such that for all $\psi\in H^{-1}(\R^n)$,
$$
\tilde L(\psi)=\langle \psi,\overline{u}\rangle_{\R^n}, \quad \textrm{and}\quad \|u\|_{H^1_{\rm scl}(\R^n)}\lesim \frac{1}{h^m}\|v\|_{L^2(\Omega)}. 
$$
Let us now show that $\mathcal L_{\varphi} u=v$ in $\Omega$. To that end, let $w\in C_0^\infty(\Omega)$. Then 
$$
\langle \mathcal L_{\varphi} u,\overline{w}\rangle_\Omega=\langle u,\overline{\mathcal L_{\varphi}^*w}\rangle_{\R^n} =\overline{\tilde L(\mathcal L_{\varphi}^*w)}=\overline{\langle w,\overline{v}\rangle_\Omega}=\langle v,\overline{w}\rangle_\Omega. 
$$
The proof is complete. 
\end{proof}

\section{Complex geometric optics solutions}\label{section on CGOs}
Let $\varphi$ be a limiting Carleman weight in an admissible manifold $(M,g)$. We will construct solutions to $\mathcal L_{g,X,q}u = 0$ in $M$ of the form
\begin{equation}\label{cgo form}
u=e^{-(\varphi+i\psi)/h}(a+r),
\end{equation}
where $a$ is an amplitude, $r$ is a correction term which is small when $h>0$ is small, and $\psi$ is a real valued phase.

Set $\rho=\varphi+i\psi$ for the complex valued phase. Consider the conjugated operator $P_\rho=e^{\rho/h}h^{2m}\mathcal L_{g,X,q}e^{-\rho/h}$, which has the following expression
$$
P_\rho=(-h^2\Delta_g-|\nabla \rho|_g^2+h\Delta_g \rho+2h\nabla \rho)^m+h^{2m}X-h^{2m-1}\<-\nabla \rho,X\>_g+h^{2m}q.
$$
Here and in what follows, the norm $|\cdot|_g^2$ and the inner product $\<\cdot,\cdot\>_g$ are extended to complex valued tangent vectors by
$$
\<\zeta,\eta\>_g=\<\Re\zeta,\Re\eta\>_g-\<\Im\zeta,\Im\eta\>_g+i(\<\Re\zeta,\Im\eta\>_g+\<\Im\zeta,\Re\eta\>_g),\quad |\zeta|_g^2=\<\zeta,\zeta\>_g.
$$
Since $m\ge 2$, in order to get
$$
e^{\varphi/h}h^{2m}\mathcal{L}_{g,X,q}(e^{-\varphi/h}a)=\mathcal O(h^{m+1}),
$$
in $L^2(M)$, we should choose $\rho$ satisfying  the following eikonal equation
\begin{equation}\label{eikonal}
|\nabla \rho|_g^2=0\quad\text{in }M,
\end{equation}
and choose $a\in C^\infty(M)$ satisfying the following transport equation
\begin{equation}\label{transport}
(2\nabla \rho+\Delta_g \rho)^m a=0\quad\text{in }M.
\end{equation}

Recall that $(M,g)$ is conformally embedded in $\R\times (M_0,g_0)$, where $(M_0,g_0)$ is some simple $(n-1)$-dimensional manifold. If necessary, we replace $M_0$ with a slightly larger simple manifold. Therefore, we can and shall assume that for some simple $(D,g_0)\subset\subset (M_0^{\rm int},g_0)$ one has
\begin{equation}\label{in larger manifold}
(M,g)\subset\subset (\R\times D^{\rm int},g)\subset (\R\times M_0^{\rm int},g).
\end{equation}
Note that $\R\times M_0$ has global coordinate chart in which the metric $g$ has the following form
\begin{equation}\label{admissible metric form}
g(x)=c(x)\left(\begin{matrix}
1&0\\
0&g_0(x')
\end{matrix}\right),
\end{equation}
where $c>0$ and $g_0$ is simple. A natural choice of the limiting Carleman weight is $\varphi(x)=x_1$. Then the equation \eqref{eikonal} for the complex valued phase $\rho$ becomes
$$
|\nabla \psi|^2=\frac{1}{c},\quad \p_{x_1}\psi=0.
$$
This equation will be solved using special coordinates on $(M,g)$. This is based on the so-called polar coordinates on the transversal simple manifold $(M_0,g_0)$. Let $\omega\in D$ be such that $(x_1,\omega)\notin M$ for all $x_1$. Points of $M$ have the form $x=(x_1,r,\theta)$ where $(r,\theta)$ are polar normal coordinates in $(D,g_0)$ with center $\omega$. That is, $x'=\exp_\omega^D(r\theta)$ where $r>0$ and $\theta\in S^{n-2}$. In terms of these coordinates the metric $g$ has the form
$$
g(x_1,r,\theta)=c(x_1,r,\theta)\left(\begin{matrix}
1&0&0\\
0&1&0\\
0&0&m(r,\theta)
\end{matrix}\right),
$$
where $m$ is a smooth positive definite matrix.

We solve \eqref{eikonal} by simply taking $\psi(x)=\psi_\omega(x)=r$. Thus, the complex valued phase has the form $\rho=x_1+ir$ and its gradient is $\nabla \rho=\frac{2}{c}\overline{\p}$, where
$$
\overline{\p}=\frac{1}{2}\(\de{}{x_1}+i\de{}{r}\).
$$

Next, we solve transport equation \eqref{transport}. In the coordinates $(x_1,r,\theta)$ equation \eqref{transport} becomes
$$
\(\frac{4}{c}\overline{\p}+\frac{1}{c}\log\frac{|g|}{c^2}\)^m a=0.
$$
Consider $a$ as the function having the following form
$$
a=|g|^{-1/4}c^{1/2}a_0(x_1,r,\theta)b(\theta)
$$
where $b$ is smooth and $a_0$ is such that $\overline{\p}a_0=ca_1$ for some $a_1$ satisfying $\overline{\p}a_1=0$.

Note that \eqref{cgo form} will be a solution for $\mathcal L_{g,X,q}u=0$ if $P_\rho (a+hr)=0$. Then, with the choice of $\varphi$ and $\psi$ made above, this equation is equivalent to the following
$$
e^{\varphi/h} h^{2m}\mathcal L_{g,X,q} e^{-\varphi/h}(e^{-i\psi/h}hr)=-e^{-i\psi/h}(h^{2m}\mathcal L_{g,X,q}a+h^{2m-1}\<\nabla \rho,X\>_g a).
$$
This will be solved by using Proposition~\ref{Solvability result}. We find $r\in H^1(M)$ satisying
$$
\|r\|_{H^1_{\rm scl}(M)}=\mathcal O(1).
$$
The discussion of this section can be summarized in the following proposition.
\begin{Proposition}\label{existence of CGOs}
Assume that $(M,g)$ satisfies \eqref{in larger manifold} and \eqref{admissible metric form}, and let $m\ge 2$ be an integer. Suppose that $X$ is a $W^{1,\infty}$ vector field on $M$ and $q\in L^\infty(M)$. Let $\omega\in D$ such that $(x_1,\omega)\notin M$ for all $x_1$. If $(r,\theta)$ are polar normal coordinates in $(D,g_0)$ with center $\omega$, then the equation
$$
\mathcal L_{g,X,q}u=0\quad\text{in}\quad M
$$
has a solution of the form
$$
u=e^{-\frac{1}{h}(\varphi+i\psi)}(|g|^{-1/4}c^{1/2}a_0(x_1,r,\theta)b(\theta)+hr),
$$
where $\overline{\p}a_0=ca_1$ for some $a_1$ depending on $(x_1,r)$ and satisfying $\overline{\p}a_1=0$, $b$ is smooth and the remainder term $r\in H^1(M)$ such that $\|r\|_{H^1_{\rm scl}(M)}=\mathcal O(1)$.
\end{Proposition}

\begin{Remark}\label{remark 1}{\rm
In fact, we need complex geometric optics solutions belonging to $H^{2m}(M)$. Such solutions can be obtained in the following way. Extend $X$ and $q$ smoothly to $\R\times M_0$. By elliptic regularity, the complex geometric optics solutions constructed as above in $M_0$ will belong to $H^{2m}(M)$.
}\end{Remark}

\begin{Remark}\label{remark 2}{\rm
It is easy to check that if $a_0$ depends only on $(x_1,r)$ and satisfies $\overline{\p}a_0=0$, then the equation $\mathcal L_{g,X,q}u=0$ in $M$ has a solution as in Proposition~\ref{existence of CGOs}.
}\end{Remark}

\section{Proof of Theorem~\ref{main th}}\label{Proof}
Let $(M,g)$ be an admissible manifold and let $m\ge 2$ be an integer. The first ingredient in the proof of Theorem~\ref{main th} is a standard reduction to a larger compact manifold with boundary.
\begin{Proposition}\label{extension of Cauchy data set}
Let $M,M_1$ be compact manifolds with boundary such that $M\subset\subset M_1$, and let $m\ge 2$ be an integer. Assume that $X_1,X_2$ are $W^{1,\infty}$ vector fields on $M$ and $q_1,q_2 \in L^\infty(M)$. Suppose that
$$
X_1=X_2,\quad q_1=q_2\quad\text{in}\quad M_1\setminus M.
$$
If $C_{g,X_1,q_1}^M=C_{g,X_2,q_2}^M$, then $C_{g,X_1,q_1}^{M_1}=C_{g,X_2,q_2}^{M_1}$, where $C_{g,X_j,q_j}^{M_1}$ denotes the set of the Cauchy data for $\mathcal L_{g,X_j,q_j}$ in $M_1$, $j=1,2$.
\end{Proposition}
\begin{proof}
Let $u\in H^{2m}(M_1)$ be a solution of $\mathcal{L}_{g,X_1,q_1}u=0$ in $M_1$. Since $C_{g,X_1,q_1}^M=C_{g,X_2,q_2}^M$, there exists $v\in H^{2m}(M)$, solving  $\mathcal{L}_{g,X_2,q_2}v=0$ in $M$, and satisfying $\gamma v=\gamma u$ in $\p M$ and $\tilde \gamma v=\tilde \gamma u$ in $\p M$.  Setting
$$
v_1=\begin{cases} v & \text{in }M,\\
u & \text{in }M_1\setminus M,
\end{cases}
$$
we get $v_1\in H^{2m}(M_1)$ and $\mathcal{L}_{g,X_2,q_2}v_1=0$ in $M_1$. Thus, $C_{g,X_1,q_1}^{M_1}\subset C_{g,X_2,q_2}^{M_1}$. Exactly the same way but in the other direction finishes the proof.
\end{proof}

The second ingredient is the derivation of the following integral identity based on the assumption that $C_{g,X_1,q_1}^M=C_{g,X_2,q_2}^M$.
\begin{Proposition}\label{main int identity prop}
Let $(M,g)$ be a compact Riemannian manifold with boundary, and let $m\ge 2$ be an integer. Assume that $X_1,X_2$ are $W^{1,\infty}$ vector fields on $M$ and $q_1,q_2 \in L^\infty(M)$. If $C_{g,X_1,q_1}=C_{g,X_2,q_2}$, then
$$
\int_M [\<X_1-X_2,v\nabla u\>_g+(q_1-q_2)uv]\,d\Vol_g(x)=0,
$$
for any $u,v\in H^{2m}(M)$ satisfying $\mathcal L_{g,-X_1,-\div_g X_1+q_1}v=0$ and $\mathcal L_{g,X_2,q_2}u=0$ in $M$.
\end{Proposition}
\begin{proof}
We will use the following consequence of the Green's formula, see \cite{Grubb},
\begin{equation}\label{Green}
(\mathcal L_{g,X_1,q_1}u,v)_{L^2(M)}=(u,\mathcal L_{g,X_1,q_1}^*v)_{L^2(M)}
\end{equation}
for all $u,v\in H^{2m}(M)$ such that $\gamma u=\gamma v=0$, where $\mathcal L_{g,X_1,q_1}^*=\mathcal L_{g,-X_1,-\div_g X_1+q_1}$.
\\

Now, let $u,v\in H^{2m}(M)$ be such that $\mathcal L_{g,-X_1,-\div_g X_1+q_1}v=0$ and $\mathcal L_{g,X_2,q_2}u=0$ in $M$. The hypothesis that $\mathcal N_{g,X_1,q_1}=\mathcal N_{g,X_2,q_2}$ implies the existence of $\tilde u\in H^{2m}(M)$ such that $\mathcal L_{g,X_1,q_1}\tilde u=0$ and $\gamma \tilde u=\gamma u$, $\widetilde{\gamma}\tilde u=\widetilde{\gamma}u$. We have
$$
\mathcal L_{g,X_1,q_1}(u-\tilde u)=(X_1-X_2)\tilde u+(q_1-q_2)\tilde u.
$$
Using \eqref{Green}, this implies the result.
\end{proof}

According to hypothesis, that $X_1=X_2$ in $(\R\times M_0)\setminus M^{\rm int}$. We also extend $q_1$ and $q_2$ to $\R\times M_0$ by zero outside $M^{\rm int}$. Let, as in Section~\ref{section on CGOs}, $(D,g_0)\subset\subset (M_0^{\rm int},g_0)$ be simple such that $(M,g)\subset\subset (\R\times D^{\rm int},g)\subset (\R\times M_0^{\rm int},g)$. Let $(M_1,g)$ be also admissible and simply connected such that $(M,g)\subset \subset (M_1^{\rm int},g)$ and $(M_1,g)\subset\subset (\R\times D^{\rm int},g)$. According to Proposition~\ref{extension of Cauchy data set}, we know that $C_{g,X_1,q_1}^{M_1}=C_{g,X_2,q_2}^{M_1}$ is true.

According to Proposition~\ref{main int identity prop} the following integral identity holds for all $u,v\in H^{2m}(M_1)$ satisfying $\mathcal L_{g,X_2,q_2}u=0$ and $\mathcal L_{g,-X_1,-\div_g X_1+q_1}v=0$ in $M_1$, respectively:
\begin{equation}\label{main int identity}
\int_{M_1} [\<X_1-X_2,v\nabla u\>_g+(q_1-q_2)uv]\,d\Vol_g(x)=0.
\end{equation}

The main idea of the proof of Theorem~\ref{main th} is to use the integral identity \eqref{main int identity} with $u,v\in H^{2m}(M)$ being complex geometric optics solutions for the equations $\mathcal L_{g,X_2,q_2}u=0$ and $\mathcal L_{g,-X_1,-\div_g X_1+q_1}v=0$ in $M_1$, respectively. We use Proposition~\ref{existence of CGOs}, Remark~\ref{remark 1} and Remark~\ref{remark 2} to choose solutions of the form
\begin{align*} 
u&=e^{-\frac{1}{h}(x_1+ir)}(|g|^{-1/4}c^{1/2}e^{i\lambda(x_1+ir)}b(\theta)+hr_1),\\
v&=e^{\frac{1}{h}(x_1+ir)}(|g|^{-1/4}c^{1/2}+hr_2),
\end{align*}
where $\lambda\in\R$ and $\|r_j\|_{H^1_{\rm scl}(M_1)}=\mathcal O(1)$, $j=1,2$. Substituting these solutions in \eqref{main int identity}, multiplying the resulting equality by $h$ and letting $h\to 0$, we get
$$
\lim_{h\to 0}\int_{M_1}\<X_1-X_2,\nabla \rho\>_g\,uv\,d\Vol_g(x)=0,
$$
where $\rho=x_1+ir$. Let us rewrite the integral in $(x_1,r,\theta)$ coordinates. Write $X=X_1-X_2$, and let $X^{\flat}$ be a $1$-form dual to $X$. Let $X_{x_1}^\flat$ and $X_r^\flat$ denote the components of $X^\flat$ in the $x_1$ and $r$ coordinates. Then
\begin{equation} \label{integral equals zero}
\int_\R\int_{M_{1,x_1}}(X_{x_1}^\flat+iX_r^\flat)e^{i\lambda(x_1+ir)}b(\theta)\,dr\,d\theta\,dx_1=0,
\end{equation}
where $M_{1,x_1}=\{(r,\theta):(x_1,r,\theta)\in M_1\}$. Since $X_1=X_2$ in $(\R\times M_0)\setminus M^{\rm int}$, we may assume that the integral is over $\R\times D$. Taking $x_1$-integral inside gives
$$
\int_{S^{n-2}}\int e^{-\lambda r}\(\int_\R e^{i\lambda x_1}(X_{x_1}^\flat+iX_r^\flat)(x_1,r,\theta)\,dx_1\)\,dr\,d\theta=0.
$$
Define
$$
f(x')=\int_\R e^{i\lambda x_1} X_{x_1}^\flat(x_1,x')\,dx_1,\quad \alpha(x')=\sum_{j=2}^n\(\int_\R e^{i\lambda x_1} X_j^\flat(x_1,x')\,dx_1\)\,dx^j.
$$
Then $f\in W^{1,\infty}(D)$ and $\alpha$ is a $1$-form which is $W^{1,\infty}$ on $D$, and the integral identity above can be rewritten as
$$
\int_{S^{n-2}}\int e^{-\lambda r}[f(\gamma_{w,\theta}(r))+i\alpha(\dot\gamma_{w,\theta}(r))]\,dr\,d\theta=0,
$$
where $\gamma_{w,\theta}$ is a geodesic in $(D,g_0)$ issued from the point $\omega$ in the direction $\theta$. For $\omega\in \p D$, the integral above is related to the attenuated ray transform of function $f$ and $1$-form $i\alpha$ in $D$ with constant attenuation $-\lambda$. Therefore, by varying the point $\omega$ in Proposition~\ref{existence of CGOs} on $\p D$ and using Proposition~\ref{ray transform} in Section~\ref{Attenuated ray transform}, for small enough $\lambda$, we have $f=-\lambda p$ and $\alpha=-idp$ where $p\in W^{1,\infty}(D)$ and $p|_{\p D}=0$. The definition of $\alpha$ and analyticity of the Fourier transform imply that 
$$
\partial_k X^\flat_j - \partial_j X^\flat_k = 0, \quad j,k = 2,\ldots,n.
$$
Also 
$$
\int e^{i\lambda x_1} (\partial_j X^\flat_1 - \partial_1 X^\flat_j)(x_1,x') \,dx_1 = \partial_j f + i\lambda \alpha_j = 0,
$$
showing that $dX^\flat=0$ in $M_1$. Since $M_1$ is simply connected, there is $\phi\in W^{2,\infty}(M_1)$ such that $\phi|_{\p M_1}=0$ and $X=\nabla \phi$.

Since $X=X_1-X_2$ in the neighborhood of the boundary $\p M_1$, we conclude that $\phi$ is a constant, say $c\in \C$, on $\p M_1$. Therefore, considering $\phi-c$, we may and will assume that $\phi=0$ on $\p M_1$. Since $X_1=X_2$ in $(\R\times M_0)\setminus M^{\rm int}$, we also may and shall assume that $\phi$ is zero outside $M_1$. In particular, $\phi$ is compactly supported.

Next, we show that $X_1=X_2$. For this, using Proposition~\ref{existence of CGOs} and Remark~\ref{remark 1}, consider
\begin{align*}
u&=e^{-\frac{1}{h}(x_1+ir)}(|g|^{-1/4}c^{1/2}e^{i\lambda(x_1+ir)}b(\theta)+hr_1),\\
v&=e^{\frac{1}{h}(x_1+ir)}(|g|^{-1/4}c^{1/2}a_0+hr_2),
\end{align*}
where $a_0$ satisfies $\overline{\p}a_0=c$. Such $a_0$ can be constructed using Cauchy's integral formula in \cite{FaK} as
$$
a_0(x_1,r,\theta)=a_0(\rho,\theta)=\frac{1}{2\pi}\int_{\mathcal B}\frac{c(z,\theta)}{z-\rho}\,dz\wedge\,d\overline{z},\quad\text{for all }\theta\in S^{n-2},
$$
where $\rho=x_1+ir$, $\mathcal B$ is a bounded domain in the upper half plane $\H\subset \C$ such that the map $\mathcal B\times S^{n-2}\to \R\times M_0$, $(x_1,r,\theta)\mapsto (x_1,\exp_{\omega}^D(r \theta))$ covers $M_1$ and the boundary $\p\mathcal B$ is piecewise smooth. Here and in what follows, $\omega\in D$ such that $\omega\in M_1$ in Proposition~\ref{existence of CGOs}.

Substituting these solutions and $X_1-X_2=\nabla \phi$ in \eqref{main int identity}, multiplying the resulting equality by $h$ and letting $h\to 0$, we get
$$
\lim_{h\to 0}\int_{M_1}\<\nabla \phi,\nabla \rho\>_g\,uv\,d\Vol_g(x)=0,
$$
where $\rho=x_1+ir$. Rewriting the integral in $(x_1,r,\theta)$ coordinates and taking $x_1$-integral inside, we obtain
$$
2\int_{S^{n-2}}\(\int_0^\infty\int_\R\overline{\p}\phi\, a_0 e^{i\lambda(x_1+ir)}b(\theta)\,dx_1\,dr\)\,d\theta=0.
$$
Since $\phi$ is compactly supported, integrating by parts, in $(x_1,r)$, gives
\begin{equation}\label{pre attenuated ray}
\begin{aligned}
0&=-\int_{S^{n-2}}\(\int_0^\infty\int_{\R}\overline{\p}\phi\, a_0 e^{i\lambda(x_1+ir)}b(\theta)\,dx_1\,dr\)\,d\theta\\
&=\int_{S^{n-2}}\(\int_0^\infty\int_{\R}\phi\,\overline{\p} a_0 e^{i\lambda(x_1+ir)}b(\theta)\,dx_1\,dr\)\,d\theta\\
&=\int_{S^{n-2}}\(\int_0^\infty\int_{\R}\phi c\, e^{i\lambda(x_1+ir)}b(\theta)\,dx_1\,dr\)\,d\theta.
\end{aligned}
\end{equation}
Set
$$
\Phi_\lambda(r,\theta)=\int_{\R}\phi c\, e^{i\lambda x_1}\,dx_1,
$$
i.e. $\Phi_\lambda$ is the Fourier transform of $\phi c$ in $x_1$-variable. Then \eqref{pre attenuated ray} can be written as
$$
\int_{S^{n-2}}\int e^{-\lambda r}\Phi_\lambda(\gamma_{\omega,\theta}(r))b(\theta)\,dr\,d\theta=0.
$$
By varying the point $\omega$ in Proposition~\ref{existence of CGOs} on $\p D$ and using \cite[Lemma~5.1]{DKS}, for small enough $\lambda$, we have $\Phi_\lambda= 0$. Since $\phi c$ is compactly supported, its Fourier transform $\Phi_\lambda$ is analytic. Therefore we obtain, $\phi =0$ which shows that $X_1=X_2$.

To show that $q_1=q_2$, consider \eqref{main int identity} with $X_1=X_2$ which becomes
\begin{equation}\label{last main int identity}
\int_{M_1} (q_1-q_2)uv\,d\Vol_g(x)=0
\end{equation}
holds for all $u,v\in H^{2m}(M_1)$ satisfying $\mathcal L_{g,X_2,q_2}u=0$ and $\mathcal L_{g,-X_1,-\div_g X_1+q_1}v=0$ in $M_1$, respectively. Use Proposition~\ref{existence of CGOs}, Remark~\ref{remark 1} and Remark~\ref{remark 2} to choose solutions of the form
\begin{align*}
u&=e^{-\frac{1}{h}(x_1+ir)}(|g|^{-1/4}c^{1/2}e^{i\lambda(x_1+ir)}b(\theta)+hr_1),\\
v&=e^{\frac{1}{h}(x_1+ir)}(|g|^{-1/4}c^{1/2}+hr_2),
\end{align*}
where $\lambda\in\R$ and $\|r_j\|_{H^1_{\rm scl}(M_1)}=\mathcal O(1)$, $j=1,2$. Substituting these solutions in \eqref{last main int identity} and letting $h\to 0$, we get
$$
\int_\R \int_{M_{1,x_1}} e^{i\lambda(x_1+ir)}(q_1-q_2)c(x_1,r,\theta)b(\theta)\,dr\,d\theta\,dx_1=0.
$$
Taking $x_1$-integral inside and varying $b$ gives
$$
\int_{S^{n-2}}\int_0^\infty \int_\R e^{i\lambda(x_1+ir)}(q_1-q_2)c(x_1,r,\theta)\,dx_1\,dr\,d\theta=0.
$$
Set
$$
Q_\lambda(r,\theta)=\int_{\R}(q_1-q_2) c\, e^{i\lambda x_1}\,dx_1,
$$
i.e. $Q_\lambda$ is the Fourier transform of $(q_1-q_2) c$ in $x_1$-variable. Then, as in the case of $\Phi_\lambda$, one can show that $Q_\lambda=0$ for all $\lambda$ small enough. We have extended $q_1$ and $q_2$ to $\R\times M_0$ by zero outside $M^{\rm }$, which implies that $q_1-q_2$ is compactly supported. Hence, $Q_\lambda$ is analytic. This together with $Q_\lambda=0$ for all $\lambda$ small enough, allows us to conclude that $q_1=q_2$.

\section{Attenuated ray transform}\label{Attenuated ray transform}
The aim of this section is to prove the following proposition which was used in the proof of Theorem~\ref{main th}. We will closely follow the arguments in \cite{DKS}.

\begin{Proposition}\label{ray transform}
Let $(D,g_0)$ be an $(n-1)$-dimensional simple manifold. Let $f\in L^\infty(D)$ and $\alpha$ be a $1$-form which is $L^\infty$ on $D$. Consider the integrals
$$
\int_{S^{n-2}}\int_0^{\tau(\omega,\theta)}[f(\gamma_{\omega,\theta}(r))+\alpha_k(\gamma_{\omega,\theta}(r))\dot\gamma_{\omega,\theta}^k(r)]e^{-\lambda r}b(\theta)\,dr\,d\theta,
$$
where $(r,\theta)$ are polar normal coordinates in $(D,g_0)$ centered at some $\omega\in\p D$, and $\tau(\omega,\theta)$ is the time when the geodesic $r\mapsto (r,\theta)$ exits $D$. If $|\lambda|$ is sufficiently small, and if these integrals vanish for all $\omega\in \p D$ and all $b\in C^\infty(S^{n-2})$, then there is $p\in W^{1,\infty}(D)$ with $p|_{\p D}=0$ such that $f=-\lambda p$ and $\alpha=d p$.
\end{Proposition}

This is related to the injectivity of attenuated ray transform acting on function and $1$-form on $D$. Let us introduce some notions and facts; see \cite{Shar} for more details. By $SD$ we will denote its unit sphere bundle $SD:=\{(x,v)\in TD: |v|_{g_0(x)}=1\}$. On the boundary of $D$, we consider the set of inward and outward unit vectors defined as
\begin{align*}
\p_+ SD&=\{(x,v)\in SD:x\in\p D,\langle v,\nu(x)\rangle_{g_0(x)}\le0\},\\
\p_- SD&=\{(x,v)\in SD:x\in\p D,\langle v,\nu(x)\rangle_{g_0(x)}\ge0\},
\end{align*}
where $\nu$ is the unit outer normal to $\p D$. The geodesics entering $D$ can be parameterized by $\p_+SD$. For any $(x,v)\in SD$ the first non-negative exit time of the geodesic $\gamma_{x,v}$, with $x=\gamma_{x,v}(0)$, $v=\dot\gamma_{x,v}(0)$, will be denoted as $\tau(x,v)$. Simplicity assumption guarantees that $\tau(x,v)$ is finite for all $(x,v)\in SD$. We also
write $\phi_t(x,v)=(\gamma_{x,v}(t),\dot\gamma_{x,v}(t))$ for the geodesic flow.

We endow the unit sphere bundle $SD$ with its usual Liouville (local product) measure $d\Sigma^{2n-3}$, and endow the bundle $\p_+SD$ with its standard measure $d\Sigma^{2n-4}$. By $d\sigma_x$ we denote the measure on $S_x D$.

Let $f$ be a function and $\alpha$ be a $1$-form on $D$. The geodesic ray transform of $f$ and $\alpha$, with constant attenuation $-\lambda$, is defined as
$$
T_\lambda[f,\alpha](x,v)=\int_0^{\tau(x,v)} [f(\gamma_{x,v}(t))+\alpha_k(\gamma_{x,v}(t))\dot\gamma_{x,v}^k(t)]e^{-\lambda t}\,dt,\quad (x,v)\in \p_+ SD.
$$

In Propostion~\ref{ray transform}, if $f$ and $\alpha$ were a continuous function and $1$-form, respectively, one could choose $b(\theta)$ to approximate a delta function at fixed angles $\theta$ and obtain that
$$
\int_0^{\tau(\omega,\theta)}[f(\gamma_{\omega,\theta}(r))+\alpha_k(\gamma_{\omega,\theta}(r))\dot\gamma_{\omega,\theta}^k(r)]e^{-\lambda r}\,dr=0
$$
for all $\omega\in \p D$ and all $\theta\in S^{n-2}$. This would imply that
$$
T_\lambda[f,\alpha](x,v)=0\quad \text{for all }(x,v)\in \p_+SD.
$$
We will use the following result from \cite[Theorem~7.1]{DKSU}.
\begin{Proposition}\label{inj of attenuated smooth case}
Let $(D,g_0)$ be a compact simple manifold with smooth boundary. There exists $\varepsilon>0$ such that the following assertion holds for a real number $\lambda$ with $|\lambda|<\varepsilon$: If $f\in C^\infty(D)$ and $\alpha$ be a smooth $1$-form on $D$, then $T_\lambda[f,\alpha](x,v)=0$ for all $(x,v)\in \p_+ SD$ implies the existence of $p\in C^\infty(D)$ with $p|_{\p D}=0$ such that $f=-\lambda p$ and $\alpha=d p$.
\end{Proposition}

The previous argument together with the above theorem proves Proposition~\ref{ray transform} for smooth $f$ and $\alpha$. However, this requires $f$ and $\alpha$ to be $C^\infty$-smooth in $D$ and it is not obvious how to do this when $f$ and $\alpha$ are $L^{\infty}$ on $D$. We resolve this problem by using duality and the ellipticity of the normal operator $T_\lambda^*T_\lambda$.
\medskip

In the space of functions on $\p_+SD$ define the inner product
$$
(h,h')_{L^2_\mu(\p_+ SD)}:=\int_D h\,h'\, d\mu
$$
where $d\mu(x,v)=\<v,\nu\>_{g_0(x)}d\Sigma^{2n-4}$. Denote the corresponding Hilbert space and the norm by $L^2_\mu(\p_+ SD)$ and $\|\cdot\|_{L^2_\mu(\p_+ SD)}$, respectively. We will also write
$$
h_\psi(x,v)=h(\phi_{-\tau(x,-v)}(x,v)),\quad (x,v)\in SD,
$$
for $h\in C^\infty(\p_+ SD)$.
\medskip

If $F$ is a notation for a function space ($C^k$, $L^p$, $H^k$, etc.), then we will denote by $\mathcal F(D)$ the corresponding space of pairs $[f,\alpha]$ with $f$  a function and $\alpha$ a 1-form on $D$. In particular, $\mathcal L^2(D)$ is the space of square integrable pairs $[f,\alpha]$, and we endow this space with the inner product
$$
([f,\alpha],[f',\alpha'])_{\mathcal L^2(D)}=\int_D \(ff'+\<\alpha,\alpha'\>_{g_0}\)\,d\Vol_{g_0}.
$$
\begin{Lemma}\label{adjoint of T_lambda}
If $f\in C^\infty(D)$, $\alpha$ is a smooth $1$-form on $D$ and $h\in C^\infty_0((\p_+ SD)^{\rm int})$, then
$$
(T_\lambda[f,\alpha],h)_{L^2_\mu(\p_+ SD)}=([f,\alpha],T_\lambda^* h)_{\mathcal L^2(D)},
$$
where $T_\lambda^* h$ is a pair defined as
$$
T^*_\lambda h(x)=\left[\int_{S_x D} h_\psi(x,v)\, e^{-\lambda \tau(x,-v)}\,d\sigma_x(v),\int_{S_x D}v^k h_\psi(x,v)\, e^{-\lambda \tau(x,-v)}\,d\sigma_x(v)\right].
$$
\end{Lemma}
\begin{proof}
By Santal\'o formula (see \cite{Shar} or \cite{DPSU})
\begin{align*}
(T_\lambda[f,\alpha],&h)_{L^2_\mu(\p_+ SD)}\\
&=\int_{\p_+SD}\int_0^{\tau(x,v)} [f(\gamma_{x,v}(t))+\alpha_k(\gamma_{x,v}(t))\dot\gamma_{x,v}^k(t)]e^{-\lambda t}\,dt\, h(x,v)\,d\mu(x,v)\\
&=\int_{SD}[f(x)+\alpha_k(x)v^k]h_\psi(x,v)e^{-\lambda\tau(x,-v)}\,d\Sigma^{2n-3}(x,v)\\
&=\int_{D}f(x)\(\int_{S_x D} h_\psi(x,v)\, e^{-\lambda \tau(x,-v)}\,d\sigma_x(v)\)\,d\Vol_g(x)\\
&\quad+\int_{D}\alpha_k(x)\(\int_{S_x D}v^k h_\psi(x,v)\, e^{-\lambda \tau(x,-v)}\,d\sigma_x(v)\)\,d\Vol_g(x).
\end{align*}
This proves the statement.
\end{proof}

\begin{proof}[Proof of Proposition~\ref{ray transform}]
First, we extend $(D,g_0)$ to a slightly larger simple manifold and to extend both $f$ and $\alpha$ by zero. Then $f$ and $\alpha$ are still in $L^\infty$, and in particular in $L^p$ for all $1<p<\infty$. In this way we can assume that both $f$ and $\alpha$ are compactly supported in $D^{\rm int}$.
\medskip

We let $b$ also depend on $\omega$ and change notations to write the assumption in the form
$$
\int_{S_x D}\int_0^{\tau(x,v)}e^{-\lambda t}[f(\gamma_{x,v}(t))+\alpha_k(\gamma_{x,v}(t))\dot\gamma_{x,v}^k(t)]b(x,v)\,dt\,d\sigma_x(v)=0
$$
for all $x\in \p D$ and $b\in C^\infty_0((\p_+ SD)^{\rm int})$. Let $\widetilde D$ be a compact submanifold of $D$ with boundary such that $(\widetilde D,g_0)$ is also simple, $\widetilde D\subset \subset D^{\rm int}$ and supports of $f$ and $\alpha$ are compact subsets of $\widetilde D^{\rm int}$. Note that $\alpha$ is $L^\infty$ on $D$ (and in particular being in $L^2$ on $D$) implies that in particular $\delta \alpha\in H^{-1}(\widetilde D)$. Then we obtain the solenoidal decomposition $\alpha=\alpha^s+dp$ on $\widetilde D$, where $\delta_{g_0}\alpha^s=0$ and $p\in H^1_0(\widetilde D)$ with $-\Delta_{g_0}p=\delta\alpha$. Here $\delta_{g_0}\alpha=\nabla_{g_0}^i\alpha_i$, where $\nabla_{g_0}$ is the covariant derivative corresponding to the metric $g_0$. Extend $p$ to $D$ by zero, so that $p\in H^1_0(D)$ with $p=0$ in $D\setminus \widetilde D$. An integration by parts shows that we have
\begin{equation}\label{some identity}
\int_{S_x D}\int_0^{\tau(x,v)}e^{-\lambda t}[f(\gamma_{x,v}(t))+\lambda p(\gamma_{x,v}(t))+\alpha^s_k(\gamma_{x,v}(t))\dot\gamma_{x,v}^k(t)]b(x,v)\,dt\,d\sigma_x(v)=0
\end{equation}
for all $x\in \p D$ and $b\in C^\infty_0((\p_+ SD)^{\rm int})$. 
Next, we make the choice $b(x,v)=h(x,v)\mu(x,v)$ for $h\in C^\infty_0((\p_+ SD)^{\rm int})$ and integrate \eqref{some identity} over $\p D$ and get
$$
(T_\lambda[f+\lambda p,\alpha^s],h)_{L^2_\mu (\p_+ SD)}=0.
$$
We are now in the same situation as in the proof of Lemma~\ref{adjoint of T_lambda}, and using the Santal\'o formula implies
$$
([f+\lambda p,\alpha^s],T^*_\lambda h)_{\mathcal L^2(D)}=0
$$
for all $h\in C^\infty_0((\p_+ SD)^{\rm int})$. Note that the last integral is absolutely convergent because $f\in L^\infty(D)$ and $\alpha$ is $1$-form which is $L^{\infty}$ on $D$, and also the previous steps are justified by Fubini's theorem.

It remains to choose $h=T_\lambda[\varphi,\beta]$ for $\varphi\in C^\infty_0(D^{\rm int})$ and $\beta$ being $C^\infty_0$-smooth $1$-form in $D^{\rm int}$, to obtain that
$$
([f+\lambda p,\alpha^s],T_\lambda^* T_\lambda [\varphi,\beta])_{\mathcal L^2(D)}=0.
$$
Since $T_\lambda^* T_\lambda$ is self-adjoint, we have
$$
(T_\lambda^* T_\lambda[f+\lambda p,\alpha^s],[\varphi,\beta])_{\mathcal L^2(D)}=0
$$
for all $\varphi\in C^\infty_0(D^{\rm int})$ and for all $C^\infty_0$-smooth $1$-form $\beta$ in $D^{\rm int}$. Therefore, $T_\lambda^* T_\lambda[f+\lambda p,\alpha^s]=0$. By \cite[Proposition~1]{HS}, $T_\lambda^*T_\lambda$ is an elliptic pseudodifferential operator of order $-1$ in $D^{\rm int}$. Here, ellipticity of $T_\lambda^* T_\lambda$ is in the sense that whenever $f',\alpha'$ are in $L^2(D^{\rm int})$ and $T_\lambda^* T_\lambda[f',\alpha']=0$ and $\delta_{g_0}\alpha'=0$, then $f',\alpha'$ are smooth. Since $f+\lambda p$ and $\alpha^s$ were compactly supported in $D^{\rm int}$, this implies that $f+\lambda p$ and $\alpha^s$ are smooth and compactly supported in $D^{\rm int}$. Hence $f+\lambda p$ and $\alpha^s$ are smooth in $D$ and compactly supported in $D^{\rm int}$. 
Now we can use the argument for smooth $f$ and $\alpha$ given above, together with Proposition~\ref{inj of attenuated smooth case} to conclude that $f=-\lambda p-\lambda \psi$ and $\alpha=\alpha^s+dp=d\psi+dp$ for some $\psi\in C^\infty(D)$ with $\psi|_{\p D}=0$. To finish the proof, it remains to show that $p\in W^{1,\infty}(D)$. But this is clear from $dp=\alpha-\alpha^s$ and from $\alpha$ is $L^{\infty}$ on $D$ and $\alpha^s$ is $C^\infty$ on $D$.
\end{proof}

\section{Boundary Determination and proof of Theorem \ref{main th2}}\label{section on boundary determination}

In this part we show boundary determination of the vector field $X$. For the generality of the statement we will assume the knowledge of the Cauchy data set $C_{g,X,q}$. It is easy to see that when $0$ is not a Dirichlet eigenvalue of $\mathcal{L}_{g,X,q}$, knowledge of the Cauchy data set is equivalent to knowledge of the Dirichlet-to-Neumann map $N_{g,X,q}$. Moreover, we can determine not only the boundary values of $X$, but also the boundary values of $q$. This is the following proposition.

\begin{Proposition} \label{bdry_deter}
Let $(M,g)$ be admissible, and let $m\geq 2$ be an integer. Suppose that $X$ is a $C^\infty$ vector field on $M$ and $q\in C^{\infty}(M)$. Then the knowledge of the Cauchy data set $C_{g,X,q}$ determines the boundary values of $X$ and the boundary values of $q$.
\end{Proposition}

To prove this proposition, it suffices to show that for any $p\in\partial M$, $C^{M}_{g,X,q}$ determines $X(p)$ and $q(p)$. In the following we will consider this local problem.\\

Fix a point $p\in\partial M$ and let $(x',x_n)$ be the boundary normal coordinates near $p$, where $x'=(x_1,\dots,x_{n-1})$. In these coordinates $\partial M$ corresponds to $\{x_n=0\}$, the vector field $X$ becomes the differential operator $X=X^{j}\frac{\partial}{\partial x_j}$, and the metric tensor can be written as
$$g=g_{\alpha\beta}dx^{\alpha}\otimes dx^{\beta}+dx^n\otimes dx^n.$$
Here the in the following we use the convention that Greek indices run from $1$ to $n-1$ and Roman indices from $1$ to $n$. Denote $D_j=\frac{1}{i}\frac{\partial}{\partial x_j}$, then the Laplace-Beltrami operator in the boundary normal coordinates takes the form
\begin{equation} \label{Laplace_normal}
-\Delta_g=D^2_n+iE(x)D_n +Q_2(x,D_{x'})+Q_1(x,D_{x'})
\end{equation}
with $E, Q_1, Q_2$ given by
\begin{align}
E(x)= & \frac{1}{2}g_{\alpha\beta}\partial_n g^{\alpha\beta}, \label{E} \vspace{1ex}\\
Q_2(x,D_{x'})= & g^{\alpha\beta}D_{\alpha}D_{\beta}, \label{Q2} \vspace{1ex} \\
Q_1(x,D_{x'})= & -i(\frac{1}{2}g^{\alpha\beta}\partial_{\alpha}(\log|g|)+\partial_{\alpha}g^{\alpha\beta})D_{\beta}. \label{Q1}
\end{align}
\\

Next we would like to write the $2m$ order equation
\begin{equation}\label{single_eqn}
\mathcal{L}_{g,X,q}u=(-\Delta_g)^{m}u+Xu+qu=0 \quad\quad\text{ in } M, \quad m\geq 2
\end{equation}
as a second order system.  To this end, introduce
$$u_1=u, \quad u_2=(-\Delta_g) u, \quad \dots, \quad u_m=(-\Delta_g)^{m-1}u$$
and let $U=(u_1,\dots,u_m)^{T}$. By a standard reduction, \eqref{single_eqn} can be written as a system of equations in $U$:
\begin{equation} \label{system}
\mathcal{L}_{A_{11}, A_{12} ,A_0}U:=(-\Delta_g\otimes I +iA_{11}(x,D_{x'})+iA_{12}(x)D_n+A_0(x))U=0 \quad\text{in}\quad M,
\end{equation}
where $I$ is the $m\times m$ identity matrix, $A_{11}(x,D_{x'})$, $A_{12}(x)$ and $A_0(x)$ are defined by
$$A_{11}(x,D_{x'}):=
\left(\begin{array}{cccc}
	0 & 0 & \dots & 0 \\
	\vdots & \vdots & \dots & \vdots \\
	0 & 0 & \dots & 0 \\
	X^\alpha(x)D_\alpha & 0 & \dots & 0 
\end{array}\right),
\quad
A_{12}(x):=
\left(\begin{array}{cccc}
	0 & 0 & \dots & 0 \\
	\vdots & \vdots & \dots & \vdots \\
	0 & 0 & \dots & 0 \\
	X^n(x) & 0 & \dots & 0 
\end{array}\right),
$$
\vspace{1ex}
$$
A_0(x):=\left(\begin{array}{ccccc}
	0 & -1 & 0 & \dots & 0 \\
	0 & 0 & -1 & \dots & 0 \\
	\vdots & \vdots & \vdots & \dots & \vdots \\
	0 & 0 & 0 & \dots & -1 \\
	q(x) & 0 & 0 & \dots & 0 
\end{array}\right).
$$
The associated Cauchy data set to the system \eqref{system} is
$$C_{A_{11}, A_{12} ,A_0}:=\{(U|_{\partial M}, \partial_{\nu}U|_{\partial M}): \mathcal{L}_{A_{11}, A_{12} ,A_0}U=0 \; \text{ in } M, \quad U\in (H^{2}(M))^{m}\}.$$
It is easy to see that $C_{A_{11}, A_{12} ,A_0}$ and $C_{g,X,q}$ are mutually determined, hence it suffices to show $C_{A_{11}, A_{12}, A_0}$ determines $A_{11}$, $A_{12}$ and $A_0$ at $p\in\partial M$. \\

The following result gives a factorization of the operator $\mathcal{L}_{A_{11}, A_{12}, A_0}$. Similar techniques are employed in \cite{DKSU, KLU, LU, NSU}.
\begin{Proposition}
There is a matrix-valued pseudodifferential operator $B(x,D_{x'})$ of order $1$ in $x'$, depending smoothly on $x_n$, such that
\begin{equation} \label{factorization}
\mathcal{L}_{A_{11}, A_{12} ,A_0}=(D_n\otimes I+iE(x)\otimes I+iA_{12}(x)-iB(x,D_{x'}))(D_n\otimes I+iB(x,D_{x'}))
\end{equation}
modulo a smoothing operator. Moreover, the principle symbol of the operator $B(x,D_{x'})$ is $-\sqrt{Q_{2}(x,\xi')}I$. Here $E(x)$ and $Q_2(x, D_{x'})$ are given by \eqref{E} and \eqref{Q2} respectively.
\end{Proposition}
\begin{proof}
Plug \eqref{Laplace_normal} into \eqref{system} we have
\begin{align*}
\mathcal{L}_{A_{11}, A_{12} ,A_0}= & (D^2_n+iE(x)D_n +Q_2(x,D_{x'})+Q_1(x,D_{x'}))\otimes I 	\vspace{1ex}\\
 & +iA_{11}(x,D_{x'})+iA_{12}(x)D_n+A_0(x).
\end{align*}
Comparing this expression with \eqref{factorization} gives the following constrains on $B(x,D_{x'})$ modulo a smoothing operator:
\begin{align}
 & B^{2}(x,D_{x'})+i[D_n\otimes I, B(x,D_{x'})]-E(x)B(x,D_{x'})-A_{12}(x)B(x,D_{x'})  \nonumber \\
= & Q_2(x,D_{x'})\otimes I +Q_1(x,D_{x'})\otimes I+iA_{11}(x,D_{x'})+A_0(x). \label{operator_equal}
\end{align}
Let $b(x,\xi')$ be the full symbol of $B(x,D_{x'})$, then \eqref{operator_equal} implies on the level of symbols that
\begin{equation} \label{symbol_equal}
\sum_{|\alpha|\geq 0}\frac{1}{\alpha !}\partial^{\alpha}_{x'}b D^{\alpha}_{\xi'}b + \partial_n b -E(x)b -A_{12}(x)b= Q_2(x,\xi')I+Q_1(x,\xi')I+iA_{11}(x,\xi')+A_0(x).
\end{equation}
Let $b\sim\sum_{j\leq 1}b_j$ where $b_j(x,\xi')$ is an $m\times m$ matrix with entries homogeneous of degree $j$ in $\xi'=(\xi_1,\dots,\xi_{n-1})$. Collecting the terms homogeneous of degree $2$ in \eqref{symbol_equal} yields
$$b^{2}_{1}(x,\xi')=Q_{2}(x,\xi')I,$$
from which we can choose
\begin{equation}\label{homo_two}
b_1(x,\xi')=-\sqrt{Q_{2}(x,\xi')}I.
\end{equation} 
Collecting the terms homogeneous of degree $1$ in \eqref{symbol_equal} yields
\begin{equation} \label{homo_one}
b_0b_1+b_1 b_0+\sum_{|\alpha|=1}\partial^{\alpha}_{x'}b_1 \partial^{\alpha}_{\xi'}b_1 +\partial_n b_1-E(x)b_1-A_{12}(x)b_1=Q_{1}(x,\xi')I+iA_{11}(x,\xi').
\end{equation}
Since $b_1(x,\xi')$ has been determined above, $E(x)$ and $Q_{1}(x,\xi')$ are known from \eqref{E} \eqref{Q1}, by some elementary linear algebra there exists a unique $b_0(x,\xi')$ satisfying this identity. Next collecting the terms of homogeneous of degree $0$ in \eqref{symbol_equal} implies
\begin{align} 
b^{2}_{0}+b_{1}b_{-1}+b_{-1}b_1+\sum_{|\alpha|=1}\partial^{\alpha}_{x'}b_1 \partial^{\alpha}_{\xi'}b_0 + \sum_{|\alpha|=1}\partial^{\alpha}_{x'}b_0 \partial^{\alpha}_{\xi'}b_1 + \sum_{|\alpha|=2}\frac{1}{2}\partial^{\alpha}_{x'}b_1 \partial^{\alpha}_{\xi'}b_1  \nonumber \\
+\partial_n b_0-E(x)b_0 -A_{12}(x)b_0=A_0(x), \label{homo_zero} 
\end{align}
From which we can solve for $b_{-1}(x,\xi')$. In general, the term $b_{j}(x,\xi')$ can be determined by considering the terms homogeneous of degree $j+1$ in \eqref{symbol_equal}. This completes the proof.
\end{proof}


\begin{proof}[Proof of Proposition \ref{bdry_deter}]
Using a similar argument as in \cite[Proposition 1.2]{LU}, we conclude that the Cauchy data set $C_{A_{11}, A_{12}, A_0}$ determines the operator $B(x',0,D_{x'})$ modulo a smoothing operator. Consequently each $b_{j}|_{x_n=0}$ is determined, $j\leq 1$. It follows from \eqref{homo_one} that the following expression is determined by the Cauchy data set $C_{A_{11}, A_{12}, A_0}$:
$$-X^{n}|_{x_n=0}\sqrt{Q_2(x',0,\xi')}+iX^{\alpha}|_{x_n=0}\xi_{\alpha}, \quad\quad \xi'\in\mathbb{R}^{n-1}.$$
Varying $\xi'$ determines $X^n|_{x_n=0}$ and $X^\alpha|_{x_n=0}$, $\alpha=1,\dots,n-1$. Evaluating \eqref{homo_zero} on $\{x_n=0\}$ shows that the Cauchy data set $C_{A_{11}, A_{12}, A_0}$ determines $A_0|_{x_n=0}$, hence $q|_{x_n=0}$.
\end{proof}

\begin{proof}[Proof of Theorem \ref{main th2}]
If $0$ is not a Dirichlet eigenvalue of $\mathcal{L}_{g,X_1,q_1}$ and $\mathcal{L}_{g,X_2,q_2}$, then $N_{g,X_1,q_1}=N_{g,X_2,q_2}$ implies $C_{g,X_1,q_1}=C_{g,X_2,q_2}$. By Proposition \ref{bdry_deter} we conclude that $X_1=X_2$ on $\partial M$. The result then follows from Theorem \ref{main th}.
\end{proof}

\medskip

\section{Proof of Theorem \ref{main th3}}\label{section on connected boundary}

We will follow the argument of \cite[Theorem 1.3]{KLU} and \cite[Theorem 4]{DKSU}. Proceeding as in the proof of Theorem \ref{main th},

We can derive the following integral identity (see \eqref{integral equals zero})
\begin{equation}
\int_M (X_{x_1}^\flat+iX_r^\flat)e^{i\lambda(x_1+ir)}b(\theta)\,dr\,d\theta\,dx_1=0.
\end{equation}
This is similar to \eqref{integral equals zero} but this time the integral is over $M$ instead of $\mathbb{R}\times M_{1,x_1}$ since we cannot extend the vector fields $X_1$ and $X_2$ any more. Varying the smooth function $b(\theta)$ leads to
$$\int_{M_\theta} (X^{b}_{x_1} + i X^{b}_{r}) e^{i\lambda(x_1+ir)} d\bar{\rho}\wedge d\rho=0 \quad\text{for all}\quad\theta\in S^{n-2}$$
where $M_\theta:=\{(x_1,r)\in\mathbb{R}^2: (x_1,r,\theta)\in M\}$ and $\rho=x_1+ir$. Define
$$
f(x')=\int_\R e^{i\lambda x_1} X_{x_1}^\flat(x_1,x')\,dx_1,\quad \alpha(x')=\sum_{j=2}^n\(\int_\R e^{i\lambda x_1} X_j^\flat(x_1,x')\,dx_1\)\,dx^j.
$$
Here we extend $X$ as zero outside of $M$ so that the integral in $x_1$ can be over $\mathbb{R}$. The above argument shows that 
$$
\int e^{-\lambda r}[f(\gamma_{\omega,\theta}(r))+i\alpha(\dot\gamma_{\omega,\theta}(r))]\,dr=0 \quad\text{for all}\quad\theta\in S^{n-2},
$$
where the $r$-integrals are integrals over geodesics $\gamma_{\omega,\theta}$ in $\pi(M)\subset M_0$. Observe that $\alpha(x')$ is $W^{1,\infty}$ on $\pi(M)$ and $f(x')$ is $L^\infty$ on $\pi(M)$. Under the assumption that $(\pi(M),g_0)$ is a simple $(n-1)$-dimensional manifold, we can apply Proposition \ref{ray transform} to $D:=\pi(M)$ and conclude that for small enough $\lambda$, we have $f=-\lambda p$ and $\alpha=-idp$ where $p\in W^{1,\infty}(\pi(M))$ and $p|_{\p \pi(M)}=0$. The definition of $\alpha$ and analyticity of the Fourier transform imply that 
$$
\partial_k X^\flat_j - \partial_j X^\flat_k = 0, \quad j,k = 2,\ldots,n \quad\text{in}\quad M^{\rm int}.
$$
Also
$$
\int e^{i\lambda x_1} (\partial_j X^\flat_1 - \partial_1 X^\flat_j)(x_1,x') \,dx_1 = \partial_j f + i\lambda \alpha_j = 0\quad\text{in}\quad M^{\rm int},
$$
showing that $dX^\flat=0$ in $M$. Since $M$ is simply connected, there exists a function $\phi$ such that $\nabla\phi=X\in W^{1,\infty}(M)$. By \cite[Theorem 4.5.12 and Theorem 3.1.7]{H} we have $\phi\in C^{1,1}(M)$.\\

Next we need to show that $\phi$ is constant on $\partial M$. In the case where we can extend $X$ to be a compactly supported $W^{1,\infty}$ vector field on a larger manifold, that $\phi$ is constant on $\partial M$ simply follows from the construction, but here we have to prove it. This is the content of the next proposition.

\begin{Proposition}
The function $\phi$ is constant on the connnected boundary $\partial M$.
\end{Proposition}

\begin{proof}

Let us start by constructing more complex geometric optics solutions. From Proposition~\ref{existence of CGOs}, Remark~\ref{remark 1} and Remark~\ref{remark 2} we can choose complex geometric optics solutions of the form
\begin{align*} 
u&=e^{-\frac{1}{h}(x_1+ir)}(|g|^{-1/4}c^{1/2}a_0(x_1,r,\theta)b(\theta)+hr_1),\\
v&=e^{\frac{1}{h}(x_1+ir)}(|g|^{-1/4}c^{1/2}+hr_2),
\end{align*}
where $\bar{\partial}a_0=0$, $\lambda\in\R$ and $\|r_j\|_{H^1_{\rm scl}(M_1)}=\mathcal O(1)$, $j=1,2$. Note that in the previous construction we choose $a_0=e^{i\lambda (x_1 + ir)}$ but this time we need more $a_0$'s. Substituting these solutions in \eqref{main int identity}, multiplying the resulting equality by $h$ and letting $h\to 0$, we get
$$
\lim_{h\to 0}\int_{M}\<X,\nabla \rho\>_g\,uv\,d\Vol_g(x)=0,
$$
where $\rho=x_1+ir$ and $X=X_1-X_2$. Recall that $X:=\nabla\phi$. Insert the above complex geometric optics solutions yields
$$
\int_M \bar{\partial}\phi a_0(x_1,r,\theta)b(\theta)\,dr\,d\theta\,dx_1=0.
$$
Varying the smooth function $b(\theta)$ leads to
$$\int_{M_\theta} \bar{\partial}\phi a_0 d\bar{\rho}\wedge d\rho=0\quad\text{for all}\quad\theta\in S^{n-2}.$$
Integrating by parts and using that $\bar{\partial} a_0=0$ gives
\begin{equation} \label{dbareqn}
\int_{\partial M_\theta} \phi a_0 d\rho=0\quad\text{for all}\quad\theta\in S^{n-2}
\end{equation}
and for every $a_0$ with $\bar{\partial} a_0=0$.

On the other hand, noticing that in solving the eikonal equation \eqref{eikonal}, we may choose $\varphi=x_1$ but $\psi=-r$. Then we can construct complex geometric optics solutions of the form
\begin{align*} 
u&=e^{-\frac{1}{h}(x_1-ir)}(|g|^{-1/4}c^{1/2}a_0(x_1,r,\theta)b(\theta)+hr_1),\\
v&=e^{\frac{1}{h}(x_1-ir)}(|g|^{-1/4}c^{1/2}+hr_2),
\end{align*}
where $\partial a_0=0$, $\lambda\in\R$ and $\|r_j\|_{H^1_{\rm scl}(M_1)}=\mathcal O(1)$, $j=1,2$. Here 
$$\partial=\frac{1}{2} \left( \frac{\partial}{\partial x_1} - i \frac{\partial}{\partial r}\right).$$
Using a similar argument as in the preceding paragraph we can derive
$$\int_{\partial M_\theta} \phi \tilde{a}_0 d\bar{\rho}=0\quad\text{for all}\quad\theta\in S^{n-2}$$
and for every $\tilde{a}_0$ with $\partial \tilde{a}_0=0$. In particular we can choose $\tilde{a}_0=\bar{a}_0$ where $a_0$ solves $\bar{\partial} a_0=0$. Then taking complex conjugate gives
\begin{equation} \label{deqn}
\int_{\partial M_\theta} \bar{\phi} a_0 d\rho=0\quad\text{for all}\quad\theta\in S^{n-2}.
\end{equation}
Combining \eqref{dbareqn} and \eqref{deqn} we see
$$
\int_{\partial M_\theta} \Re \phi \,a_0 d\rho=0, \quad\quad \int_{\partial M_\theta} \Im \phi \,a_0 d\rho=0
$$
for all $a_0$ with $\bar{\partial} a_0=0$. Using the argument in \cite[Section 5]{DKSU2007} implies that $\Re\phi|_{\partial M_\theta}=F|_{\partial M_\theta}$ for some non-vanishing holomorphic function $F$ on $M_\theta$. Observing that $\Im F$ is a harmonic function in $M_\theta$ and $\Im F|_{\partial M_\theta}=0$, we conclude that $F$ is real-valued and hence is constant on each connected component of $\partial M_\theta$. Varying $\theta$ shows that $\Re \phi$ is constant along $\partial M$. Likewise we can show $\Im \phi$ is also constant along $\partial M$.
\end{proof}

\begin{proof}[Proof of Theorem \ref{main th3}]
Since $\phi=c$ for some constant $c$ along $\partial M$, replacing $\phi$ by $\phi-c$ if necessary, we may assume $\phi=0$ on $\partial M$. The rest part of the proof is the same as that of Theorem \ref{main th}.
\end{proof}

\bigskip\bigskip

\textbf{Acknowledgement:} The authors would like to thank Prof. Katya Krupchyk for her suggestions on an earlier version of this paper. The authors are also deeply grateful to Prof. Gunther Uhlmann for his generous support relating to this project.


\begin{thebibliography}{ABC}

\bibitem{Cal} A. Calderon, \emph{On an inverse boundary value problem}, Comput. Appl. Math. {\bf 25} (2006), 133--138.

\bibitem{DPSU} N. S. Dairbekov, G. P. Paternain, P. Stefanov, G. Uhlmann, \emph{The boundary rigidity problem in the presence of a magnetic field}, Adv. Math. {\bf 216} (2007), 535--609.

\bibitem{DKSU} D. Dos Santos Ferreira, C. Kenig, M. Salo, G. Uhlmann, \emph{Limiting Carleman weights and anisotropic inverse problems}, Invent. Math. {\bf 178} (2009), no. 1, 119--171.

\bibitem{DKSU2007} D. Dos Santos Ferreira, C. Kenig, J. Sj\"{o}strand, G. Uhlmann, \emph{Determining a magnetic Schr\"{o}dinger operator from partial Cauchy data}, Comm. Math. Phys., {\bf 271} (2007), 467--488.

\bibitem{DKS} D. Dos Santos Ferreira, C. Kenig, M. Salo, \emph{Determining an unbounded potential from Cauchy data in admissible geometries}, Comm. PDE {\bf 38} (2013), no. 1, 50--68.

\bibitem{FaK} H. M. Farkas, I. Kra, \emph{Riemann Surfaces}, Second Edition, Graduate Texts in Mathematics, 71, Springer-Verlag, New-York, 1992.

\bibitem{GGS} F. Gazzola, H.-C. Grunau, G. Sweers, \emph{Polyharmonic boundary value problems}, Springer-Verlag, Berlin, 2010.

\bibitem{Grubb} G. Grubb, \emph{Distributions and operators}, volume 252 of Graduate Texts in Mathematics. Springer, New York, 2009.

\bibitem{H} L. H\"{o}rmander, \emph{The analysis of linear partial differential operators. I. Distribution theory and Fourier analysis}, Classics in Mathematics. Springer-Verlag, Berlin, 2003.

\bibitem{HS} S. Holman, P. Stefanov, \emph{The weighted Doppler transform}, Inverse Probl. Imaging {\bf 4} (2010), 111--130.

\bibitem{Isakov} V. Isakov, \emph{Completeness of products of solutions and some inverse problems for PDE}, J. Differential Equations {\bf 92} (1991), no. 2, 305--316.

\bibitem{KLU} K. Krupchyk, M. Lassas, G. Uhlmann, \emph{Inverse Boundary value Problems for the Perturbed Polyharmonic Operator}, Transactions AMS, {\bf 366} (2014), 95--112.

\bibitem{KLU2} K. Krupchyk, M. Lassas, G. Uhlmann, \emph{Determining a first order perturbation of the biharmonic operator by partial boundary measurements}, J. Funct. Anal., {\bf 262} (2012), 1781--1801.

\bibitem{KSU} C. Kenig, J. Sj\"ostrand, G. Uhlmann, \emph{The Calder\'on problem with partial data}, Ann. of Math. (2) {\bf 165} (2007), no. 2, 567--591.

\bibitem{KU} K. Krupchyk, G. Uhlmann, \emph{Uniqueness in an inverse boundary problem for a magnetic Schr\"odinger operator with a bounded magnetic potential}, Comm. Math. Phys., {\bf 327} (2014), 993--1009.

\bibitem{LU} J. Lee, G. Uhlmann, \emph{Determining anisotropic real-analtic conductivities by boundary measurements}, Comm. Pure Appl. Math., {\bf 42} (1989), no. 8, 1097--1112.

\bibitem{Michel} R. Michel, \emph{Sur la rigidit\'e impos\'ee par la longueur des g\'eod\'esiques}, Invent. Math., \textbf{65} (1981), 71--83.

\bibitem{NSU} G. Nakamura, Z. Sun, G. Uhlmann, \emph{Global identifiability for an inverse problem for the Schr\"odinger equation in a magnetic field}, Math. Ann. {\bf 303} (1995), no. 3, 377--388.

\bibitem{S} M. Salo, \emph{Inverse problems for nonsmooth first order perturbations of the Laplacian}, Ann. Acad. Sci. Fenn. Math. Diss. {\bf 139} (2004).

\bibitem{SaU} M. Salo, G. Uhlmann, \emph{The attenuated ray transform on simple surfaces}, J. Diff. Geom. {\bf 88} (2011), no. 1, 161--187.

\bibitem{Shar} V. A. Sharafutdinov, \emph{Integral geometry of tensor fields}, Inverse and Ill-Posed Problems Series, VSP, Utrecht, 1994.

\bibitem{Sun} Z. Sun, \emph{An inverse boundary value problem for Schr\"odinger operators with vector potentials}, Trans. Amer. Math. Soc. {\bf 338} (1993), no. 2, 953--969.

\bibitem{SyU} J. Sylvester, G. Uhlmann, \emph{ A global uniqueness theorem for an inverse boundary value problem}, Ann. of Math. (2) {\bf 125} (1987), no. 1, 153--169.

\bibitem{Tol} C. Tolmasky, \emph{Exponentially growing solutions for nonsmooth first-order perturbations of the Laplacian}, SIAM J. Math. Anal. {\bf 29} (1998), no. 1, 116--133.

\bibitem{Y} Y. Yang, \emph{Determining the first order perturbation of a bi-harmonic operator on bounded and unbounded domains from partial data}, J. Differ. Equations {\bf 257} (2014), 3607--3639.
\end{thebibliography}
\end{document}